\documentclass[a4paper,11pt,twoside]{amsart}
\usepackage[english]{babel}
\usepackage[utf8]{inputenc}

\usepackage[a4paper,inner=3cm,outer=3cm,top=4cm,bottom=4cm,pdftex]{geometry}
\usepackage{fancyhdr}
\usepackage{titlesec}
\usepackage{color}
\usepackage{bold-extra}
\usepackage{ mathrsfs }
\titleformat{\section}{\normalfont\scshape\centering}{\thesection}{1em}{}
  \titleformat{\subsection}{\bfseries}{\thesubsection}{1em}{}

\usepackage{comment}
\usepackage{graphics}
\usepackage{aliascnt}
\usepackage[pdftex,citecolor=green,linkcolor=red]{hyperref}

\usepackage{amsmath}
\usepackage{amsfonts}
\usepackage{amssymb}
\usepackage{amsthm}
\usepackage{comment}
\usepackage{mathtools}

\newtheorem{theorem}{Theorem}[section]

\newtheorem{lemma}[theorem]{Lemma}
\newtheorem{proposition}[theorem]{Proposition}
\theoremstyle{definition}

\newtheorem{remark}[theorem]{Remark}

\numberwithin{equation}{section}

\renewcommand{\Re}{\textnormal{Re}}
\renewcommand{\Im}{\textnormal{Im}}
\renewcommand\P{\mathbf{P}}

\newcommand\n{\mathbf{n}}

\begin{document}

\title{On the M\"obius function in all short intervals}

\author{Kaisa Matom\"{a}ki}
\address{Department of Mathematics and Statistics, University of Turku, Turku, Finland}
\email{ksmato@utu.fi}

\author{Joni Ter\"{a}v\"{a}inen}
\address{Mathematical Institute, University of Oxford, Oxford, UK}
\email{joni.teravainen@maths.ox.ac.uk}

\begin{abstract}
We show that, for the M\"obius function $\mu(n)$, we have
\begin{align*}
\sum_{x < n\leq x+x^{\theta}}\mu(n)=o(x^{\theta})
\end{align*}
for any $\theta>0.55$. This improves on a result of Motohashi and Ramachandra from 1976, which is valid for $\theta>7/12$. Motohashi and Ramachandra's result corresponded to Huxley's $7/12$ exponent for the prime number theorem in short intervals. The main new idea leading to the improvement is using Ramar\'e's identity to extract a small prime factor from the $n$-sum. The proof method also allows us to improve on an estimate of Zhan for the exponential sum of the M\"obius function as well as some results on multiplicative functions and almost primes in short intervals. 
\end{abstract}

\maketitle

\section{Introduction}
Let $\Lambda(n)$ and $\mu(n)$ denote the von Mangoldt and M\"obius functions. In 1972 Huxley~\cite{huxley} proved that the prime number theorem holds on intervals of length $H \geq x^{7/12+\varepsilon}$, i.e.
\begin{align}\label{eq0.5}
\sum_{x < n\leq x+H}\Lambda(n)=(1+o(1))H\quad  \text{for $H\geq x^{7/12+\varepsilon}$.}
\end{align}
Soon after Huxley's work, Motohashi \cite{motohashi} and Ramachandra~\cite{rama} independently adapted the proof to the case of the M\"obius function. In fact, Ramachandra handled a larger class of sequences arising as Dirichlet series coefficients of products of Dirichlet $L$-functions, their powers, logarithms, and derivatives (a class of sequences whose most notable representatives are $\mu(n)$ and $\Lambda(n)$), showing that for such sequences we have an asymptotic formula for their sums over short intervals $[x,x+H]$ of length $H \geq x^{\theta}$ for any $\theta > 7/12 = 0.5833\dotsc$. 

The only improvement to the results of Huxley and Motohashi and Ramachadra is Heath-Brown's~\cite{hb-1988} result that one can obtain an asymptotic formula for intervals of length $H\geq x^{7/12-\varepsilon(x)}$ for any $\varepsilon(x)$ tending to $0$ at infinity.

In this paper we show that in various instances, including the M\"obius function but not the von Mangoldt function, the exponent $x^{7/12+\varepsilon}$ can be improved to $x^{0.55+\varepsilon}$. For the M\"obius function our result is
\begin{theorem}\label{theo_mobius} Let $\theta>0.55$ and $\varepsilon>0$ be fixed. Then, for $x$ large enough and $H\geq x^{\theta}$, we have
\begin{align}\label{eq0}
\sum_{x < n\leq x+H}\mu(n)=O\left(\frac{H}{(\log x)^{1/3-\varepsilon}}\right).
\end{align}
\end{theorem}

Note that even under the Riemann hypothesis, one can only get such results for $\theta > 1/2$ (see e.g. \cite[Section 10.5]{iw-kow}), so our theorem moves a long-standing record significantly closer to a natural barrier. 

The main new idea leading to the improvement is using an identity attributed to Ramar\'e (see \cite[Chapter 17.3]{opera} and formula \eqref{ramare} below) to extract a small prime factor from
the $n$–sum. We will discuss the proof ideas, limitations of different methods etc. in more detail in Section~\ref{se:Discussion}. 

Like Ramachandra's method, ours works for a wide class of multiplicative functions. In particular, we can strengthen a result proved by Ramachandra in \cite{rama} (and obtained in unpublished work of Hooley and Huxley) on sums of two squares in short intervals, which again involved the exponent $7/12$.

\begin{theorem}\label{theo_square} Let $N_0(x)$ denote the number of integers $\leq x$ that can be written as the sum of squares of two integers. Then for $\theta>0.55$  and $\varepsilon>0$ fixed, $x$ large enough and $H\geq x^{\theta}$, we have
\begin{align*}
N_0(x+H)-N_0(x)=(C+O((\log x)^{-1/6+\varepsilon}))\frac{H}{(\log x)^{1/2}},
\end{align*}
where $C=\frac{1}{\sqrt{2}}\prod_{p\equiv 3 \pmod 4}\left(1-1/p^2\right)^{-1/2}$ is the Landau--Ramanujan constant.
\end{theorem}

We can also apply our method to the $k$-fold divisor functions $\tau_k$. For $k\leq 5$, short sums of $\tau_k(n)$ over an interval $[x,x+x^{\theta_k+\varepsilon}]$ are well-understood by directly applying the fact that one obtains a large power saving in the corresponding long sums (see \cite{huxley-divisor}, \cite{kolesnik}, \cite[Theorem 13.2]{ivic} for the exponents $\theta_2=131/416$, $\theta_3=43/96$ and $\theta_k=(3k-4)/(4k)$, $4\leq k\leq 8$, respectively). 

However, for large $k$, understanding $\tau_k$ in short intervals is closely connected to the problem of understanding the von Mangoldt function $\Lambda$ in short intervals, which one can currently asymptotically estimate only on intervals around $x$ of length $\geq x^{7/12+o(1)}$. Therefore, the best value of $\theta_k$ for $k \geq 6$ in the literature is $\theta_k=7/12+\varepsilon$, coming from Ramachandra's theorem \cite{rama}. Our next theorem says that we can in fact do better for the divisor functions than for the primes.

\begin{theorem}\label{theo_divisor} Let $\tau_k(n)$ denote the $k$-fold divisor function. Then for $\theta>0.55$  and $\varepsilon>0$ fixed, $x$ large enough and $H\geq x^{\theta}$, we have
\begin{align*}
\sum_{x< n\leq x+H}\tau_k(n)=P_{k-1}(\log x)H+O(H(\log x)^{(2/3+\varepsilon)k-1}),
\end{align*}
where $P_{k-1}$ is a polynomial of degree $k-1$ that can be be calculated explicitly (see \cite[Formula (5.36) in Section II.5.4]{TenenbaumBook}).
\end{theorem}

The proof of Theorem \ref{theo_divisor} works for non-integer values of $k$ as well (including complex values), and although in those cases the function $P_{k-1}(\log x)$ will not be a polynomial anymore, it can still be expressed as a linear combination of the functions $(\log x)^{k-1-j}$ for $j\geq 0$. 

The proof of Theorem \ref{theo_mobius} is inapplicable for the corresponding problem for the von Mangoldt function, since one cannot extract small prime factors from numbers $n$ in the support of $\Lambda(n)$. Nevertheless, the proof does work for $E_2$ almost primes, that is to say numbers of the form $p_1p_2$ with $p_1,p_2$ primes. We are able to obtain an asymptotic for the count of $E_2$ numbers on intervals of length $x^{0.55+\varepsilon}$.

\begin{theorem}\label{theo_E2} Let $\theta>0.55$ be fixed. Then for $x$ large enough and $H\geq x^{\theta}$ we have
\begin{align*}
\sum_{\substack{x < n\leq x+H\\n\in E_2}}1= H\frac{\log \log x}{\log x} + O\left(H \frac{\log \log \log x}{\log x}\right).
\end{align*}
\end{theorem}

Our method can also be used for twisted sums. We demonstrate this by proving the following theorem.
\begin{theorem}\label{theo_Zhan} Let $\theta>3/5$ and $\varepsilon>0$ be fixed. Then for $x$ large enough and $H\geq x^{\theta}$ we have, uniformly for $\alpha \in \mathbb{R}$,
\begin{align*}
\sum_{x < n\leq x+H}\mu(n) e(\alpha n)=O\left(\frac{H}{(\log x)^{1/3-\varepsilon}}\right).
\end{align*}
\end{theorem}
Previously $\mu(n) e(\alpha n)$ was known to exhibit cancellations in intervals of length $H = x^\theta$ with $\theta > 5/8$, due to work of Zhan~\cite[Theorem 5]{Zhan92} from 1991. 

\section{Discussion of results, methods, and their limitations}
\label{se:Discussion}
\subsection{The case of the M\"obius function}
\label{ssec:MetMobius}
The $7/12$ exponent in Huxley's as well as in Motohashi's and Ramachandra's works is a very natural barrier: A crucial piece of information needed in Huxley's, Motohashi's and Ramachandra's proofs is a bound of the form $N(\sigma,T)\ll T^{B(1-\sigma)}$ (where $N(\sigma, T)$ is the number of zeros of the Riemann zeta function in the rectangle $\Re(s)\geq \sigma$, $|\Im(s)|\leq T$) for $T\geq 2$, $\sigma\in [1/2,1]$, with $B$ as small as possible. The best value of $B$ to date is Huxley's $B=\frac{12}{5}+o(1)$, which is the reason for the appearance of the $7/12$ exponent.

Huxley's prime number theorem \eqref{eq0.5} was proved differently by Heath-Brown in \cite{hb-identity}, but this proof also runs into serious difficulties when one tries to lower $\theta$ below $7/12$. Heath-Brown does not use zero density results but rather uses a combinatorial decomposition (Heath-Brown's identity) and mean and large value estimates for Dirichlet polynomials, but since zero density estimates are based on these, the difficulty one runs into is actually essentially the same.

Our proof of Theorem \ref{theo_mobius} manages to avoid the lack of improvements to Huxley's zero-density estimate by means of Ramar\'e's identity, which allows a more flexible combinatorial factorization of the M\"obius function than what arises from applying Heath-Brown's identity from \cite{hb-identity} alone: We will first apply Ramar\'e's identity to extract a small prime factor and then Heath-Brown's identity to the remaining long variable.

Starting with \cite{mr-annals}, Ramar\'e's identity has been successfully applied to many problems involving multiplicative functions. In particular, connected to our problem it was shown in \cite{mr-annals} that $\mu(n)$ has a sign change on every interval of the form $(x,x+Cx^{1/2}]$ with $x\geq 1$ and $C>0$ a large enough constant. Problems of the type
\begin{align}\label{eq9}
\sum_{x < n\leq x+H}1_{\mu(n)=-1}>0\quad \textrm{or}\quad \sum_{x < n\leq x+H}\Lambda(n)>0
\end{align}
are of course easier than their asymptotic counterparts \eqref{eq0} and \eqref{eq0.5}, and there are indeed various results establishing a positive lower bound for the count of primes in intervals of length shorter than $x^{7/12}$; see \cite{iw-jut}, \cite{H-BI79}, \cite{bh1} and \cite{bhp2}, among others. The record to date is due to Baker, Harman and Pintz \cite{bhp2} with the exponent $0.525$, and an earlier result of Heath-Brown and Iwaniec \cite{H-BI79} established the exponent $0.55+\varepsilon$ that we obtain here for the \emph{asymptotic} problem \eqref{eq0} rather than for the lower bound problem \eqref{eq9}. It is no coincidence that we obtain the same exponent, as our work draws a crucial lemma from theirs to handle type I/II information (see Lemma~\ref{le_hbi} below), and the proof of that lemma is not continuous in $\theta$ but crucially uses that $\theta > 0.55$.

The ultimate reason why the exponent $0.55$ is in fact the limit of our method is that when one applies Heath-Brown's identity to the M\"obius function, one needs to bound mean values of various products of Dirichlet polynomials (which are either partial sums of the Riemann zeta function or very short polynomials), and the particular case where we have five polynomials of length $x^{1/5+o(1)}$ is a case where it seems that the large values sets of the polynomials ''corresponding to the $3/4$-line'' can no longer be shown to be small enough for $\theta<0.55$ if one uses existing mean value theorems for the Riemann zeta function (such as the fourth moment or twisted moment results).

In the case of Huxley's $7/12$-result, the worst case is having six Dirichlet polynomials of length $x^{1/6+o(1)}$ but, thanks to Ramar\'e's identity, in the case of $\mu$ we can extract an additional small prime factor so that this particular configuration of polynomials can simply be dealt with a pointwise estimate, Cauchy--Schwarz and the mean value theorem for Dirichlet polynomials (see Lemma~\ref{le_bilinear} below for this argument). For primes, such a trick of extracting a small prime factor is not available.

The proof of Theorem~\ref{theo_Zhan} concerning the twisted M\"obius function follows similarly using Ramar\'e's identity to introduce a small prime variable before running Zhan's argument (which involves again Heath-Brown's identity and mean values of Dirichlet polynomials), and we will outline the proof in Section~\ref{ssec:Zhan}. The reason that one cannot go beyond $3/5$ in Theorem~\ref{theo_Zhan} is again the case where Heath-Brown's identity leads to five factors of size $x^{1/5+o(1)}$.

As we need to restrict to numbers with a small prime factor, we have to content ourselves with a rather small saving in~\eqref{eq0} (although the $1/3-\varepsilon$ exponent can be improved; see Remark \ref{rem_improve} below). In contrast, the previous methods give, for some $c >0$ and any $H \geq x^\theta$ with $\theta > 7/12$, the bound
\[
\sum_{x < n\leq x+H}\mu(n) = O \left(H \exp\left(-c \left(\frac{\log x}{\log \log x}\right)^{1/3}\right)\right)
\]
(see e.g.~\cite[Remark 4]{rama}), and a similar bound holds e.g. for the error term in Huxley's prime number theorem.

\subsection{The case of multiplicative functions and almost primes}
\label{ssec:MetOthers}
When trying to generalize Theorem~\ref{theo_mobius} to more general multiplicative functions, one needs to be careful: Unlike in the case of almost all short intervals handled in~\cite{mr-annals}, in general the short averages of multiplicative functions do not always match long averages; for more discussion on this, see Remark~\ref{rem:CounterEx} below.

Despite such limitations, we can prove the following general result from which Theorems~\ref{theo_square} and \ref{theo_divisor} immediately follow since the class of multiplicative functions under consideration in particular includes the generalized divisor functions $\tau_z(n)$ for any complex $z$, as well as the indicator of those integers that can be represented as the norm of an element in a given abelian extension $K/\mathbb{Q}$. 

\begin{proposition}\label{prop_periodic} Let $f:\mathbb{N}\to \mathbb{C}$ be a multiplicative function which is eventually periodic on the primes in the sense that, for some integers $n_0, D\geq 1$, we have $f(p)=f(q)$ whenever $p\equiv q\pmod D$ and $p,q\geq n_0$. Suppose further that $|f(n)|\leq \tau_{\kappa}(n)$ for some integer $\kappa \geq 1$. Then, for $\varepsilon>0$ fixed and $H\geq x^{0.55+\varepsilon}$ we have
\begin{align*}
\sum_{x< n\leq x+H}f(n)=\frac{H}{x}\sum_{x< n\leq 2x}f(n)+O\left(\frac{H}{\log x}\prod_{p\in [1,x]\setminus [P,Q]}\left(1+\frac{|f(p)|}{p}\right)\right),    
\end{align*}
where $P=\exp((\log x)^{2/3+\varepsilon/2})$ and $Q=x^{1/(\log \log x)^2}$.
\end{proposition}
The proof follows along similar lines as that of Theorem~\ref{theo_mobius}, and we will discuss the differences in Section~\ref{ssec:period}. The class of multiplicative functions is chosen so that a Heath-Brown type combinatorial decomposition is possible --- the same class was recently considered by Drappeau and Topacogullari in~\cite{DrappeauTopa} in the context of the generalized Titchmarsh divisor problem. 

In the proof of Theorem~\ref{theo_E2} on $E_2$ numbers the crucial fact that we use about $E_2$ numbers is that almost all of them have a small prime factor. There is a slight complication though: the small prime factors that almost all $E_2$ numbers have are in fact so small that we do not necessarily have the Vinogradov--Korobov bound for the corresponding Dirichlet polynomials, and this requires us to use slightly more delicate estimates for Dirichlet polynomials than in the proof of Theorem \ref{theo_mobius}. We will describe the proof of Theorem \ref{theo_E2} in Section \ref{ssec:E2}.

The novelty in Theorem \ref{theo_E2} is that there we get an asymptotic formula for the number of \emph{all} $E_2$ numbers in short intervals --- if one only considered $E_2$ numbers whose prime factors are of \emph{specific sizes} (say a set of the form $\{n\leq x:\,\, n=p_1p_2,\quad x^{\alpha} < p_1\leq x^{\alpha+\varepsilon}\}$ with $\alpha$ suitably chosen), one could prove an asymptotic formula for the number of these in shorter intervals.

\section{Extracting a small prime factor}
\label{sec:Ramare}

We begin by proving the combinatorial identity that we need and that is based on Ramar\'e's identity. This identity allows us to introduce a small prime variable to our sum, which will turn out to be crucial in what follows. One could alternatively use a Tur\'an--Kubilius type argument to introduce a small prime variable but that would lead to much weaker error terms.

\begin{lemma}\label{le_ramare} Let $\varepsilon>0$ be fixed, let $x$ be large enough, and let $(\log x)^{4}\leq P<Q\leq x^{o(1/\log \log x)}$, $x^{\varepsilon}\leq H\leq x$. Then we have
\begin{align}\label{eq2}
\sum_{x < n\leq x+H}\mu(n)=-\sum_{\substack{x< prn\leq x+H\\P< p\leq Q\\r\leq x^{\varepsilon/2}}}a_r\mu(n)+O\left(H\frac{\log P}{\log Q}\right),
\end{align}
with $a_r$ being an explicit sequence (given by \eqref{eq1}) that satisfies $|a_r|\leq \tau(r)$.
\end{lemma}

Note that the coefficients $a_r$ here will be harmless, due to the restrictions on their size and support.

\begin{proof}
By a standard application of the linear sieve (e.g. \cite[Corollary 6.2]{iw-kow}),
\[
\sum_{\substack{x < n \leq x+H \\ p \mid n \implies p \not \in (P, Q]}} 1 = O\left(\frac{H\log P}{\log Q}\right),
\]
so we may add to the sum on the left-hand side of \eqref{eq2} the condition $(n,\prod_{P< p\leq Q}p)>1$, obtaining
\begin{align}\label{eq8}
\sum_{x < n\leq x+H}\mu(n)=\sum_{x < n\leq x+H}\mu(n)1_{(n,\prod_{P< p\leq Q}p)>1}+O\left(\frac{H\log P}{\log Q}\right).
\end{align}
We then apply Ramar\'e's identity
\begin{align}\label{ramare}
\mu(n)1_{(n,\prod_{P< p\leq Q}p)>1}=\sum_{P< p\leq Q}\sum_{pm=n}\frac{\mu(p)\mu(m)}{\omega_{(P,Q]}(m)+1}+O(1_{p^2\mid n,\,\, p\in (P.Q]}),
\end{align}
where $\omega_{(P,Q]}(m)$ is the number of distinct prime divisors of $m$ on $(P,Q]$; this identity follows directly since $\mu$ is multiplicative and the number of representations $n=pm$ with $P<p\leq Q$ is $\omega_{(P,Q]}(n)$. The contribution of the $O(\cdot)$ term, after summing over $x < n\leq x+H$, is trivially bounded by
\begin{align*}
\ll \sum_{P< p\leq Q}\frac{H}{p^2}\ll \frac{H}{P},
\end{align*}
which can be included in the error term.

In order to decouple the $p$ and $m$ variables, we write $m$ uniquely as $m=m_1m_2$ with $m_1$ having all of its prime factors from $(P,Q]$ and $m_2$ having no prime factors from that interval. Then we see that
\begin{align*}
\sum_{x < n\leq x+H}\mu(n)=\sum_{P< p\leq Q}\sum_{\substack{x/p\leq m_1m_2\leq (x+H)/p\\p'\mid m_1\Longrightarrow p'\in (P,Q]\\p''\mid m_2\Longrightarrow p''\not \in (P,Q]}}\frac{\mu(p)\mu(m_1)\mu(m_2)}{\omega_{(P,Q]}(m_1)+1}+O\left(H\frac{\log P}{\log Q}\right).
\end{align*}
We wish to restrict the support of the $m_1$ variable. By writing $k=pm_1m_2$ (and noting that any $k$ has at most $\omega(k)$ such representations) and recalling that $Q<x^{\varepsilon/(10A\log \log x)}$ for $A\geq 10$ fixed and $x$ large enough, we see that the contribution of the terms with $m_1>x^{\varepsilon/4}$ is bounded by
\begin{align*}
\leq \sum_{\substack{x< k\leq x+H\\\omega(k)\geq 10A/4\cdot \log \log x}}\omega(k)\ll (\log x)2^{-10A/4\cdot \log \log x}\sum_{x< k\leq x+H}2^{\omega(k)}\ll \frac{H}{(\log x)^{A}},
\end{align*}
say, by Shiu's bound \cite[Theorem 1]{shiu}. This is certainly an admissible error term.

Next, we need to dispose of the coprimality condition on $m_2$ in the sums above. For this we use the fundamental lemma of the sieve (see e.g. \cite[Chapter 6]{iw-kow}). Let $\lambda_d^{+}$ and $\lambda_d^-$ be the linear upper and lower bound sieve coefficients with the level of distribution $y:=Q^{s}$ with $s=100\log \log x$. Write $\mathcal{P}(P,Q) = \prod_{P < p \leq Q} p$. Then, by  \cite[inequality (6.19)]{iw-kow}, we have 
\[
\begin{split}
\left|1_{(m_2, \mathcal{P}(P,Q))=1}-\sum_{d\mid (m_2, \mathcal{P}(P,Q))} \lambda_d^{+}\right| &=  \sum_{d\mid (m_2, \mathcal{P}(P,Q))} \lambda_d^{+} - 1_{(m_2,\mathcal{P}(P,Q))=1} \\
&\leq \sum_{d\mid (m_2, \mathcal{P}(P,Q))} \lambda_d^{+} - \sum_{d\mid (m_2, \mathcal{P}(P,Q))} \lambda_d^-.
\end{split}
\]
Hence
\begin{equation}
\begin{split}
\label{eq:Ramrem1dec}
\sum_{x < n\leq x+H}\mu(n)&= \sum_{\substack{x< pm_1dn\leq x+H\\P< p\leq Q, \, m_1\leq x^{\varepsilon/4}\\p'\mid dm_1\Longrightarrow p'\in (P,Q]}}\lambda_d^{+}\frac{\mu(p)\mu(m_1)\mu(dn)}{\omega_{(P,Q]}(m_1)+1} + O\left(H \frac{\log P}{\log Q}\right) \\
& \qquad + O\left(\sum_{\substack{x< pm_1dn\leq x+H \\ P< p\leq Q, m_1\leq x^{\varepsilon/4} \\ p'\mid dm_1\Longrightarrow p'\in (P,Q]}} (\lambda_d^{+} - \lambda_d^-) \right).
\end{split}
\end{equation}
In the last error term we can sum first over $n$ obtaining that it is at most of order
\begin{align*}
& \sum_{\substack{P< p\leq Q\\m_1\leq x^{\varepsilon/4}\\p'\mid m_1\Longrightarrow p'\in (P,Q]}} \frac{H}{m_1 p} \left(\sum_{d\mid \mathcal{P}(P,Q)} \frac{\lambda_d^{+}}{d} - \sum_{d\mid  \mathcal{P}(P,Q)} \frac{\lambda_d^-}{d}\right) + O(yQx^{\varepsilon/4}).
\end{align*}
Now, by the fundamental lemma of the sieve (see e.g. \cite[Theorem 6.1]{iw-kow}), the difference in the parentheses  is $O(e^{-s}) = O((\log x)^{-100})$, which leads to an admissible error term after summing over $m_1$ and $p$.

In the main term on the right-hand side of~\eqref{eq:Ramrem1dec} we have $\mu(dn)=\mu(d)\mu(n)$ unless $(d,n)>1$, and if the latter condition holds, there must exist a prime $q\in (P,Q]$ such that $q\mid d$ and $q\mid n$. Writing $k=pm_1dn$, and applying Shiu's bound, the contribution of the case $(d,n)>1$ is
\begin{align*}
\ll \sum_{P< q\leq Q}\sum_{\substack{x< k\leq x+H\\q^2\mid k}}\tau_4(k) \ll \sum_{P< q\leq Q}\frac{H}{q^2}(\log x)^3 \ll \frac{H}{P\log P}(\log x)^3 \ll \frac{H}{\log x},
\end{align*}
which is small enough. Thus we may replace $\mu(dn)$ in~\eqref{eq:Ramrem1dec} by $\mu(d)\mu(n)$. Defining
\begin{align}\label{eq1}
a_r:=(\lambda^{+}\mu*w\mu)(r),\quad \textrm{where}\,\, w(r):=\frac{1_{r \leq x^{\varepsilon/4}} 1_{p \mid r \implies p \in (P, Q]}}{\omega_{(P,Q]}(r)+1},
\end{align}
we see that the sequence $a_r$ is supported on $r\leq x^{\varepsilon/4}Q^{s}\leq x^{\varepsilon/2}$ and $|a_r|\leq \tau(r)$ (since $|\lambda_d^{+}|\leq 1$), so we obtain \eqref{eq2}.
\end{proof}

\begin{remark}\label{rem_f}
Let $f:\mathbb{N}\to \mathbb{C}$ be any multiplicative function satisfying $|f(n)|\leq \tau_{\kappa}(n)$ for some fixed $\kappa\geq 1$, and let $P$ and $Q$ be as above, with the additional constraint that $P\geq (\log x)^{A_{\kappa}}$ for $A_{\kappa}$ a large enough constant. Since Shiu's bound is applicable for $|f(n)|$, it is clear from the above proof that the analogous statement
\begin{align*}
\sum_{x < n\leq x+H}f(n)=\sum_{\substack{x< prn\leq x+H\\P< p\leq Q\\r\leq x^{\varepsilon/2}}}a_r f(p)f(n)+O\left(\frac{H}{\log x}\prod_{p\in [1,x]\setminus [P,Q]}\left(1+\frac{|f(p)|}{p}\right)\right),
\end{align*}
holds when $\mu$ is replaced by $f$ in definition of $a_r$ in~\eqref{eq1}.
\end{remark}

\section{The M\"obius function in short intervals}
\label{sec:MobiusProof}
\subsection{Applying Heath-Brown's identity}
\label{ssec:HB}
In what follows, we fix the choices
\begin{align*}
P=\exp((\log x)^{2/3+\varepsilon/2}),\quad Q=x^{1/(\log \log x)^2},\quad  H=x^{\theta},\quad \theta=0.55+\varepsilon,\quad k=20.
\end{align*}
It suffices to prove \eqref{eq0} with $H\asymp x^{0.55+\varepsilon}$, since the case $H\geq x^{0.55+\varepsilon}$ then follows by splitting the sum into short sums.

With Lemma \ref{le_ramare}, we can introducee a short prime variable into the sum of $\mu(n)$ over a short interval, and we now apply Heath-Brown's identity \cite{hb-identity}. Let $M(s)=\sum_{m\leq (2x)^{1/k}}\mu(m) m^{-s}$. Then we have the Dirichlet series identity
\begin{align*}
\frac{1}{\zeta(s)}=\sum_{1\leq j\leq k}(-1)^{j-1}\binom{k}{j}\zeta(s)^{j-1}M(s)^j+\frac{1}{\zeta(s)}(1-\zeta(s)M(s))^k,
\end{align*}
which on taking the coefficient of $n^{-s}$ on both sides for $n\leq 2x$ gives Heath-Brown's identity for the M\"obius function:
\begin{align*}
\mu(n)=\sum_{1\leq j\leq k}(-1)^{j-1}\binom{k}{j}1^{(*)(j-1)}*(\mu1_{[1,(2x)^{1/k}]})^{(*)j},
\end{align*}
where $f^{(*)j}$ is the $j$-fold Dirichlet convolution of $f$. Applying this to the $n$ variable on the right-hand side of \eqref{eq2}, we see that
\begin{align*}
\sum_{\substack{x< prn\leq x+H\\P< p\leq Q}}a_r\mu(n) = \sum_{j = 1}^k (-1)^{j-1}\binom{k}{j} \sum_{\substack{x< prn_1\cdots n_{2j-1}\leq x+H\\P< p\leq Q\\i\geq j\Longrightarrow n_i\leq (2x)^{1/k}}} a_r\mu(n_j)\cdots \mu(n_{2j-1}).
\end{align*}

Further splitting all the variables into dyadic intervals and adding dummy variables, we end up with a linear combination of $\ll(\log x)^{2k+2}$ sums of the form
\begin{align}\label{eq3}
\sum_{\substack{x< prn_1\cdots n_{2k-1}\leq x+H\\p\in (P_1,2P_1],r\in (R,2R], n_i\in (N_i,2N_i]  \\p\leq Q }} a_ra_1(n_1)\cdots a_{2k-1}(n_{2k-1}),
\end{align}
where
\begin{equation}\label{eq10}
\begin{split}
&P_1\in [P,Q],\quad R\in [1/2,x^{\varepsilon/2}],\quad N_1, \dotsc, N_{k-1} \in [1/2,x], \\ &
N_{k}, \dotsc, N_{2k-1} \in [1/2, (2x)^{1/k}], \quad P_1RN_1\cdots N_{2k-1}\asymp x,
\end{split}
\end{equation}
and for each $i$ we have
\begin{align*}
a_i(n)\equiv\begin{cases} 1\quad \mathrm{or}\quad 1_{n=1},\quad i\leq k-1\\
\mu(n)1_{n \leq (2x)^{1/k}} \quad \mathrm{or} \quad 1_{n=1},\quad k\leq i\leq 2k-1.
\end{cases}
\end{align*}

Recall that $k=20$. What we wish to establish is a comparison principle, which states that
\begin{align}\label{eq5}\begin{split}
&\sum_{\substack{x< prn_1\cdots n_{2k-1}\leq x+H\\p\in (P_1,2P_1],r\in (R,2R], n_i\in (N_i,2N_i]  \\ p\leq Q}} a_ra_1(n_1)\cdots a_{2k-1}(n_{2k-1})\\
&=\frac{H}{y_1}\sum_{\substack{x< prn_1\cdots n_{2k-1}\leq x+y_1\\\\p\in (P_1,2P_1],r\in (R,2R], n_i\in (N_i,2N_i] \\p\leq Q}} a_ra_1(n_1)\cdots a_{2k-1}(n_{2k-1})+O_A\left(\frac{H}{(\log x)^{A}}\right),
\end{split}
\end{align}
with $y_1=x\exp(-3(\log x)^{1/3})$ for any choices of $P_1,R,N_i$ in \eqref{eq10}. Indeed, once we have this, we can recombine these $\ll (\log x)^{2k+2}$ sums into the single sum \eqref{eq2} to obtain
\begin{align}\label{eq7}
\sum_{x < n\leq x+H}\mu(n)= \frac{H}{y_1}\sum_{x < n \leq x+y_1}\mu(n)+O\left(H\frac{\log P}{\log Q}\right),
\end{align}
and by the prime number theorem for the M\"obius function with the Vinogradov--Korobov error term (or by Ramachandra's result \cite{rama}) this becomes the statement of Theorem \ref{theo_mobius}.

\subsection{Arithmetic information}
\label{ssec:arithmetic}

Our first lemma for establishing comparisons of the form \eqref{eq5} is the type I information arising from the case where one of the $N_i$ corresponding to a smooth variable is very long.\footnote{As pointed out by the referee, this lemma would actually follow from Lemmas \ref{le_bilinear} and \ref{le_hbi} below. We have chosen to keep Lemma \ref{le_typeI} here only because it gives a more elementary way of obtaining type I information.}

\begin{lemma}\label{le_typeI}
Suppose that in the sum \eqref{eq3} we have $N_i\gg x^{0.45+2\varepsilon}$ for some $i\leq k-1$. Then \eqref{eq5} holds.
\end{lemma}

\begin{proof}
The difference of the two sums on different sides of \eqref{eq5} is of the type
\begin{align*}
\sum_{\substack{x< mn\leq x+H\\n\in (N_i,2N_i]}}b_m - \frac{H}{y_1}\sum_{\substack{x< mn\leq x+y_1\\n\in (N_i,2N_i]}}b_m
\end{align*}
with $H\asymp x^{0.55+\varepsilon}$, $N_i\gg x^{0.45+2\varepsilon}$, $y_1=x\exp(-3(\log x)^{1/3})$, and with $b_m$ a divisor-bounded sequence. Thus the result follows trivially by summing over the $n$ variable first (cf. \cite[p.128]{HarmanBook}).
\end{proof}

When the above type I information is not applicable, we move to Dirichlet polynomials in order to obtain type II and type I/II information. As usual (see for instance \cite[Chapter 7]{HarmanBook}), we may apply Perron's formula to reduce to Dirichlet polynomials.

\begin{lemma}\label{le_perron} Let $T_0=\exp((\log x)^{1/3})$ and define the Dirichlet polynomials $P(s)=\sum_{P_1 < p\leq \min\{2 P_1, Q\}}p^{-s}$, $R(s)=\sum_{R < r\leq 2R}a_r r^{-s}$, $N_i(s)=\sum_{N_i < n\leq 2N_i}a_i(n)n^{-s}$. Suppose that
\begin{align}\label{eq6}
\int_{T_0}^{x^{1+\varepsilon/10}/H}\left|P\left(\frac{1}{2}+it\right)R\left(\frac{1}{2}+it\right)N_1\left(\frac{1}{2}+it\right)\cdots N_{2k-1}\left(\frac{1}{2}+it\right)\right|\, dt\ll_{A} x^{1/2}(\log x)^{-A}
\end{align}
for all $A>0$. Then we have \eqref{eq5}.

\end{lemma}

\begin{proof}
This is a standard application of Perron's formula detailed in \cite[Lemma 7.2]{HarmanBook}.
\end{proof}

The following lemma gives us type II information where the small prime $p$ arising from Ramar\'e's identity is crucial. In that lemma and later, for a positive integer $K$, we use the notation $[K] = \{1, 2, \dotsc, K\}$.

\begin{lemma}\label{le_bilinear} Let the notation be as in Section~\ref{ssec:HB}. Suppose that there is a subset $I\subset [2k-1]$ with $\prod_{i\in I}N_i\in [x^{0.45-\varepsilon/2},x^{0.55+\varepsilon/2}]$. Then \eqref{eq5} holds.
\end{lemma}

\begin{proof}
By Lemma \ref{le_perron}, it suffices to show \eqref{eq6}. Note first that writing $\sigma = 1/(\log X)^{2/3+\varepsilon/3}$ we have by the Vinogradov--Korobov zero-free region, for any $|t| \leq x^2$,
\[
|P(1+it)| \ll P^{-\sigma} (\log x)^3 + \frac{\log x}{|t|+1}
\]
(see e.g. \cite[Proof of Lemma 2]{mr-note}). Recalling that $P\geq \exp((\log x)^{2/3+\varepsilon/2})$, this gives
\[
|P(1+it)| \ll_A (\log x)^{-A}
\] 
for any $A > 0$ and $|t| \in [T_0, x^{1+\varepsilon/10}/H]$. Now we can bound the left hand side of~\eqref{eq6} by using this bound to $P(s)$, the trivial bound to $R(s)$, and Cauchy--Schwarz and the mean value theorem (\cite[Theorem 9.1]{iw-kow}) to the remaining Dirichlet polynomials, obtaining, for any $A > 0$, a bound of
\begin{align*}
&\ll_A R^{1/2} \log R \cdot (\log x)^{-A} \\
&\quad \cdot\left( \int_{T_0}^{x^{1+\varepsilon/10}/H}\left|\prod_{i\in I}N_i\left(\frac{1}{2}+it\right)\right|^2\, dt \int_{T_0}^{x^{1+\varepsilon/10}/H}\left|\prod_{i\in [2k-1]\setminus I}N_i\left(\frac{1}{2}+it\right)\right|^2\, dt \right)^{1/2}\\
&\ll R^{1/2}P^{1/2}(\log x)^{-A} \left(x^{1+\varepsilon/10}/H+\prod_{i\in I}N_i\right)^{1/2}\left(x^{1+\varepsilon/10}/H+\prod_{i\in [2k-1]\setminus I}N_i\right)^{1/2}\\
&\ll x^{1/2}(\log x)^{-A},
\end{align*}
as desired.
\end{proof}

Finally we have the following type I/II information for trilinear sums with one smooth variable.
\begin{lemma}[Heath-Brown--Iwaniec]\label{le_hbi} Let the notation be as in Section~\ref{ssec:HB}. Suppose that there exists an index $r$ such that $[2k-1]\setminus \{r\}$ can be partitioned into two sets $I$ and $J$ such that $\prod_{i\in I}N_i\ll x^{0.46+\varepsilon/8}$ and $\prod_{i\in J}N_i\ll x^{0.46+\varepsilon/8}$. Then \eqref{eq5} holds.
\end{lemma}

\begin{proof} Since $k=20$ and $N_r\gg x/\prod_{i\in [2k-1]\setminus \{r\}}N_i\gg x^{0.079}$, we must have $r\leq k-1$ in \eqref{eq3}, so the polynomial $N_r(s)$ is a partial sum of the zeta function of the form $\sum_{N_r< n\leq 2N_r}n^{-s}$. Then we may apply a lemma of Heath-Brown and Iwaniec \cite[Lemma 2]{H-BI79} (alternatively see \cite[Lemma 10.12]{HarmanBook}) to conclude.
\end{proof} 

There would be a lot more arithmetic information available, see e.g~\cite[Section 10.5]{HarmanBook}. However, none of this handles for $\theta < 0.55$ the case where one has five smooth variables of size $x^{1/5+o(1)}$, so this additional information would not help us.

\subsection{Combining the results}
\label{ssec:combine}

Let $N_i=x^{\alpha_i}$ in \eqref{eq3} for some real numbers $0\leq \alpha_i\leq 1$. Combining Lemmas \ref{le_typeI}, \ref{le_bilinear} and \ref{le_hbi} (of which the first one is actually included in the other two), it clearly suffices to prove the following combinatorial lemma, after which we obtain the comparison \eqref{eq5} for all choices of the $N_i$, and thus obtain \eqref{eq7}.

\begin{lemma}
Let $\varepsilon>0$ be small, and let $K$ be a positive integer. Let $\alpha_1, \dotsc, \alpha_{K} \in (0, 1]$ be such that $\sum_{i=1}^K \alpha_i \in [1-\varepsilon, 1]$. Then one of the following holds.
\begin{enumerate}
\item There exists a subset $I \subset [K]$ such that $\sum_{i \in I} \alpha_i \in [0.45, 0.55]$.
\item There exists a partition $[K] = I_1 \cup I_2 \cup \{r\}$ such that one has $\sum_{i \in I_j} \alpha_i \leq 0.46$ for $j = 1, 2$.
\end{enumerate}
\end{lemma}

\begin{proof}
We can assume that there is no subset $I \subset [K]$ for which $\sum_{i \in I} \alpha_i \in [0.45, 0.55]$ since otherwise we are in case (1).

Let then $I \subset [K]$ be the subset with the largest $\sum_{i \in I} \alpha_i \leq 0.55$ (which actually must be $< 0.45$). Now, for any $r \in [K] \setminus I$ one has $\alpha_r + \sum_{i \in I} \alpha_i > 0.55$ since otherwise we contradict $I$ having the largest sum. Consequently we are in case (2) with $I_1 = I$ and $I_2 = [K] \setminus (I \cup \{r\})$ (since $\sum_{i \in I_2} \alpha_i < 1-0.55 < 0.46$). 
\end{proof}

\begin{remark}
As pointed out by the referee, instead of Heath-Brown's identity we could apply the Vaughan type identity
\[
\frac{1}{\zeta(s)} = \left(\frac{1}{\zeta(s)} - M(s)\right) (1-\zeta(s) M(s)) + 2M(s) - \zeta(s) M(s)^2
\]
with $M(s) = \sum_{n \leq x^{0.45}} \mu(n) n^{-s}$ as this gives rise only to terms that can be handled through Lemmas~\ref{le_bilinear} and~\ref{le_hbi}. However, in the proof of the more general Proposition~\ref{prop_periodic}, we will anyway have to use a decomposition that leads to similar terms as Heath-Brown's identity.
\end{remark}

\subsection{The twisted case}
\label{ssec:Zhan}
In this section we outline how our argument can be combined with that of Zhan to prove Theorem~\ref{theo_Zhan}. As Zhan, we start by introducing a rational approximation 
\[
\alpha = \frac{a}{q} + \lambda, \quad (a, q) = 1, \quad |\lambda| \leq \frac{1}{q\tau}, \quad 1 \leq q \leq \tau = H^2 x^{-1} (\log x)^{-B}
\]
for some large $B > 0$. Zhan has already proved Theorem~\ref{theo_Zhan} in the minor arc case $q > (\log x)^B$ (see~\cite[Theorem 2]{Zhan92} which is stated for the von Mangoldt function but the same proof works for the M\"obius function). Hence we can concentrate on the major arc case $q \leq (\log x)^B$. 

We have, similarly to Lemma~\ref{le_ramare},
\begin{align*}
\sum_{x < n\leq x+H}\mu(n)e(\alpha n) =-\sum_{\substack{x< prn\leq x+H\\P< p\leq Q\\r\leq x^{\varepsilon/2}}}a_r\mu(n) e(\alpha prn)+O\left(H\frac{\log P}{\log Q}\right).
\end{align*}
As in proof of Theorem~\ref{theo_mobius}, we use Heath-Brown's identity to decompose this into $\ll(\log x)^{2k+2}$ sums of the form
\begin{align*}
\sum_{\substack{x< prn_1\cdots n_{2k-1}\leq x+H\\i\geq k\Longrightarrow n_i\leq (2x)^{1/k}\\p\in (P_1,2P_1],r\in (R,2R], n_i\in (N_i,2N_i]  \\P<p\leq Q }} a_ra_1(n_1)\cdots a_{2k-1}(n_{2k-1}) e(\alpha prn_1 \dotsm n_{2k-1}),
\end{align*} 
with same notation and conditions on the variables as in Section~\ref{ssec:HB}.

Now, we would like to show the comparison principle~\eqref{eq5} with both main terms twisted by $e(\alpha prn_1 \dotsm n_{2k-1})$. The argument Zhan uses in the minor arc case (see \cite[Proof of Theorem 2]{Zhan92}) reduces this to mean values of Dirichlet polynomials through moving into character sums and using partial integration, Perron's formula and the first derivative test.

To state the required mean value result, we introduce the notation 
\[
T_1 =  4\pi(|\lambda| x + x/H) \quad \text{and} \quad F(s, \chi) = P(s, \chi)R(s, \chi)N_1(s, \chi)\cdots N_{2k-1}(s, \chi),
\]
where the Dirichlet polynomials are as in Lemma~\ref{le_perron} but twisted by $\chi$. Then, slightly modifying Zhan's argument from~\cite[Section 3, see in particular formulas (3.11)--(3.12)]{Zhan92}, noting that we have somewhat different notation, we see that it suffices to prove that, for all $A>0$,
\begin{align}
\label{eq:ZhanClaim}
\sum_{\chi \pmod q} \int_{T}^{T + x/H}\left|F\left(\frac{1}{2}+it, \chi\right)\right|\, dt \ll_{A} (qx)^{1/2}(\log x)^{-A} \quad \text{for $T \in [T_0, T_1]$}
\end{align}
and
\begin{align*}
&\sum_{\chi \pmod q} \int_{T}^{2T}\left|F\left(\frac{1}{2}+it, \chi\right)\right|\, dt \ll_{A} \frac{T}{x/H} \cdot (qx)^{1/2}  (\log x)^{-A} \quad \text{for $T \geq T_1$}.
\end{align*}
The second claim is easier than the first since all the bounds one uses for proving \eqref{eq:ZhanClaim} depend at most linearly on the length of the integration interval.

Zhan proves~\eqref{eq:ZhanClaim} for $q > (\log x)^B$. As in our proof of Theorem~\ref{theo_mobius}, he splits into three cases --- type I sums, type $\textrm{I}_2$ sums and type II sums. Zhan's type I and type $\textrm{I}_2$ estimates (\cite[Propositions 1 and 2]{Zhan92}) based on second and fourth moments of $L$-functions in short intervals work directly also for $q \leq (\log x)^B$.

Hence it suffices to show that also Zhan's type II bound \cite[Proposition 3]{Zhan92} holds in our situation, with the upper bound in \cite[(3.16)]{Zhan92} replaced by $H x^{-\varepsilon/10}$ (this replacement can be done since we only aim for intervals of length $x^{3/5+\varepsilon}$ rather than $x^{3/5} (\log x)^A$). But here we can utilize the short polynomial by using the pointwise estimate $|P(1/2+it) R(1/2+it)| \ll_A (PR)^{1/2} (\log x)^{-A}$. After that we can use Cauchy-Schwarz and mean value theorem for Dirichlet polynomials exactly as Zhan who got his saving from the estimate $1 \leq q^{1/2} (\log x)^{-A}$ which holds only in the minor arc case. \qed

\section{Multiplicative functions and almost primes in short intervals}
In this section we describe how the proof of Theorem~\ref{theo_mobius} needs to be modified to prove Proposition~\ref{prop_periodic} and Theorem~\ref{theo_E2}. 

\subsection{Eventually periodic multiplicative functions}
\label{ssec:period}

\begin{proof}[Proof of Proposition \ref{prop_periodic}.]
The first difference compared to proof of Theorem~\ref{theo_mobius} is that instead of Lemma \ref{le_ramare} we apply Remark \ref{rem_f} that generalizes it to multiplicative functions. This gives us
\[
\sum_{x<n\leq x+H}f(n)= \sum_{\substack{x< prn\leq x+H\\P< p\leq Q\\r\leq x^{\varepsilon/2}}}f(p)a_rf(n)+O\left(\frac{H}{\log x}\prod_{p\in [1,x]\setminus [P,Q]}\left(1+\frac{|f(p)|}{p}\right)\right),
\]
where $a_r=(\lambda^{+}f*wf)(r)$.

Next we provide a Heath-Brown type combinatorial decomposition for $f(n)$ (Drappeau and Topacogullari~\cite{DrappeauTopa} also provide combinatorial decompositions for $f(n)$, but we show an alternative way to obtain a suitable decomposition). Letting $K = \lfloor 1000 \log \log x\rfloor $ and $w = x^{1/K}$, we may write
\begin{align}\begin{split}
\label{eq:splitS}
&\sum_{\substack{x< prn\leq x+H\\P< p\leq Q\\r\leq x^{\varepsilon/2}}}f(p)a_rf(n)\\
&=\sum_{\substack{0 \leq \ell \leq 60\\0\leq k\leq K}} \frac{1}{\ell! k!}\sum_{\substack{x< p r m p_1 \dotsm p_k q_1 \dotsm q_{\ell} \leq x+H\\P< p\leq Q\\ w < p_i \leq x^{1/60}\\ q_i > x^{1/60}\\ p'\mid m\Longrightarrow p'\leq w}} a_r f(p)f(p_1)\cdots f(p_k)f(q_1)\cdots f(q_{\ell})f(m) + O\left(\frac{H}{w}\right).
\end{split}
\end{align}
Note that we can restrict the $m$ variable above to be $\leq x^{\varepsilon/2}$ in size, adding an acceptable error $O(H/(\log x)^{10})$ (cf. the proof of Lemma \ref{le_ramare}), so $rm$ plays just the same role as the $r$ variable in case of the M\"obius function. 

For each $b \pmod D$, let $a_b$ be such that $f(q) = a_b$ for every prime $q > n_0$ with $q \equiv b \pmod{D}$. Then, for every prime $q > \max\{D, n_0\}$,
\[
f(q) = \sum_{\chi \pmod{D}} \left( \frac{1}{\varphi(D)} \sum_{b \pmod{D}} a_b \overline{\chi(b)} \right) \chi(q) = \sum_{\chi\pmod D}c_{\chi}\chi(q),
\]
say, where $|c_{\chi}|\ll 1$.

We use this expansion for each variable $q_i$ in~\eqref{eq:splitS}. Thus we are left with obtaining the comparision principle for sums of the form
\begin{align*}
\sum_{\substack{0 \leq \ell \leq 60\\0\leq k\leq K}} \frac{1}{\ell! k!}\sum_{\substack{x< p r m p_1 \dotsm p_k q_1 \dotsm q_{\ell} \leq x+H\\P< p\leq Q\\ w < p_i \leq x^{1/60}\\ q_i > x^{1/60}\\ p'\mid m\Longrightarrow p'\leq w}} a_r f(p) f(p_1) \dotsm f(p_k) \chi_1(q_1)\cdots \chi_{\ell}(q_{\ell})f(m),  
\end{align*}
with $\chi_i$ any Dirichlet characters $\pmod D$.

For the $q_i$ variables, we introduce the von Mangoldt weight and then apply Heath-Brown's identity (e.g. with $k = 20$). We split the resulting sums as well as sums over $p$ and $r$ dyadically, getting $\ll (\log x)^{39 \ell+2}$ sums. Note that for the $\ell=0$ terms one does not need to use Heath-Brown's identity, since in those terms all the variables already have length $\leq x^{1/60}$.

If we for a while ignore the issue that $k$ (the number of primes $p_i$) is sometimes large, then we end up with sums essentially of the form \eqref{eq3}, with the $a_i(n)$ being slightly different but having the crucial property that any sequence $a_i(n)$ supported outside $[1,(2x)^{1/20}]$ is of the form $\chi(n)$ with $\chi$ a Dirichlet character $\pmod D$. 

All of the lemmas we applied in the proof of Theorem~\ref{theo_mobius} are readily available for sums of the form $\sum_{N\leq n\leq 2N}\chi(n)n^{-s}$ in addition to their unweighted counterparts. Furthermore, in the analogue of \eqref{eq7} for the function $f$ one can use on the right-hand side for example Ramachandra's result \cite{rama}, since any $f$ that we consider can be expressed as the Dirichlet series coefficients of a function of the form $\prod_{\chi\pmod{D}}L(s,\chi)^{\alpha_\chi} F_0(s)$, with $F_0(s)$ an absolutely convergent Dirichlet series for $\Re(s)>1/2$ (see e.g. \cite[proof of Lemma 2.3]{DrappeauTopa}), and hence $f$ is in Ramachandra's class of functions.

One can deal with $k$ being large by grouping the variables $p_i$ into $\leq 30$ products whose sizes are in $[x^{1/30}, x^{1/20}]$: We take $i_0 = 0$, and then define, for $j \geq 1$, $i_j$ recursively so that, for each $j$, we let $i_j$ be the first index for which $p_{i_{j-1}+1} \dotsm p_{i_j} \geq x^{1/30}$. We continue recursion until the step $h$ for which $p_{i_{h-1} +1} \dotsm p_{k} < x^{1/30}$ and write $i_h = k$ (note that we might have $i_{h-1} = k$ as well).  Necessarily $h \leq 30$, so there are less than $K^{30} = o(\log x)$ possibilities for the tuple $(i_1, \dotsc, i_{h-1})$. We can write
\begin{equation}
\label{eq:pjsplit}
\begin{split}
&\frac{1}{k!} \sum_{\substack{n = p_1 \dotsm p_k \\ w < p_i \leq x^{1/60}}} f(p_1)\cdots f(p_k) =  \sum_{h=1}^{30} \sum_{0 = i_0 < i_1 < \dotsb \leq i_h = k} \frac{(i_1-i_0)! (i_2-i_1)! \dotsm (i_h-i_{h-1})!}{i_h!} \\
& \qquad \qquad \qquad \qquad \cdot \sum_{\substack{n=v_1 \dotsm v_h \\ v_j \leq x^{1/20} \\ v_1, \dotsc, v_{h-1} \geq x^{1/30} > v_h}} b_{i_1-i_0, v_1} b_{i_2-i_1, v_2} \dotsm b_{i_{h}-i_{h-1}, v_h},
\end{split}
\end{equation}
where
\[
b_{r, v} := \sum_{\substack{p_1 \dotsm p_r = v \\ w< p_i \leq x^{1/60} \\ p_1 \dotsm p_{r-1} < x^{1/30}}} \frac{f(p_1) \dotsm f(p_{r})}{r!},
\]
and we have $|b_{r,v}| = O(\kappa^r) = O((\log x)^{1000 \log \kappa})$ --- this size bound is sufficient since we show the comparision principle with saving $(\log x)^{-A}$ for any $A > 0$. Inserting~\eqref{eq:pjsplit} into~\eqref{eq:splitS} and splitting each $v_i$ dyadically, this deals with the problem of $k$ being large, which was the only remaining issue in the proof of Theorem~\ref{theo_square}. 
\end{proof}

\begin{remark}
\label{rem:CounterEx}
The proof above crucially used the eventual periodicity of $f(p)$, and actually some conditions on $f$ must be imposed --- for any $\theta \in (0, 1)$ and any large $x$, there are multiplicative functions such that the relation
\begin{equation}
\label{eq:shortlongcomp}
\frac{1}{H}\sum_{x < n \leq x+H} f(n) = (1+o(1)) \frac{1}{x} \sum_{x < n \leq 2x} f(n)
\end{equation}
does not hold for $H = x^\theta$.

This can be demonstrated for instance by letting, for $j = 1, 2$, $f_j=f_{j,x}$ be the multiplicative function defined at prime powers by
\[
f_j(p^k) =
\begin{cases}
(-1)^j \mu(m) & \text{if $p^k \geq H$ and $mp^k \in (x, x+H]$ for some (necessarily unique) $m$}; \\
\mu(p^k) & \text{otherwise.}
\end{cases}
\]
Then
\begin{align*}
\sum_{x < n \leq x+H} (f_2(n) - f_1(n)) &= \sum_{\substack{x < p^k m \leq x+H \\ p^k \geq H}} (f_2(p^k m) - f_1(p^k m)) \\
& \geq \sum_{m \leq x^{\varepsilon}} \mu(m)^2 \sum_{x/m < p \leq (x+H)/m} 1 \gg_\theta H
\end{align*}
at least for $\theta \geq 7/12 + 2\varepsilon$ by Huxley's prime number theorem. If $\theta < 7/12 + 2\varepsilon$, in turn, we may split an interval around $x$ of length $\asymp x^{7/12+2\varepsilon}$ into intervals of length $x^{\theta}$ and note that by the pigeonhole principle we have in any case
\[
\sum_{x' < n \leq x'+H} (f_2(n) - f_1(n)) \gg H
\]
for some $x'\asymp x$. On the other hand, by Hal\'asz's theorem (see e.g. \cite[Theorem 4.5 in Section III.4.3]{TenenbaumBook}), for $j = 1, 2$ and any $x' \asymp x$,
\[
\sum_{x' < n \leq 2x'} f_j(n) = o(x'),
\]
so~\eqref{eq:shortlongcomp} cannot hold for both $f_1$ and $f_2$. If one restricts the support of $f$ to $H$-smooth numbers, then one can hope to prove~\eqref{eq:shortlongcomp} and this is subject of an on-going work of Granville, Harper, Radziwi{\l\l} and the first author.
\end{remark}

\subsection{Almost primes}
\label{ssec:E2}
The proof of Theorem \ref{theo_E2} mostly follows the arguments proof of Theorem \ref{theo_mobius}, but starts with the following simple decomposition for $E_2$ numbers:
\begin{align}\label{eq12}
\sum_{\substack{x<n\leq x+H\\n\in E_2}}1=\sum_{\substack{x<p_1p_2\leq x+H\\\exp((\log \log x)^2)\leq p_1\leq x^{\varepsilon}}} 1+O\left(H\frac{\log \log \log x}{\log x}\right)+O_{\varepsilon}\left(\frac{H}{\log x}\right).   
\end{align}
The validity of this is seen simply by using the Brun--Titchmarsh inequality to estimate the number of those $p_1p_2\in (x,x+H]$ with $p_1<\exp((\log \log x)^{2})$ or $p_1>x^{\varepsilon}$. Here we think of $\varepsilon>0$ as being fixed.

Note that an additional complication compared to the proof of Theorem \ref{theo_mobius} is that the $p_1$ variable may be as small as $\exp((\log \log x)^2)$, and thus we do not have the Vinogradov--Korobov zero-free region for the corresponding Dirichlet polynomial. Therefore, we will need to modify some steps in the proof of Theorem \ref{theo_mobius} for the current proof.

On the right-hand side of \eqref{eq12}, we replace the indicator function of the prime $p_2$ by the von Mangoldt weight $\Lambda(p_2)$ for which we have Heath-Brown's identity. We apply Heath-Brown's identity to $\Lambda(p_2)$ with $k=20$. As in Section \ref{ssec:HB}, we obtain a linear combination of $\ll (\log x)^{2k+2}$ sums of the form \eqref{eq3} (with $2k-1$ replaced by $2k$), where now $R=1/2$, 
\begin{align}\label{eq16}
a_i(n)\equiv\begin{cases} 1\quad \mathrm{or}\quad \log n\quad \mathrm{or}\quad 1_{n=1},\quad i\leq k\\
\mu(n)1_{n \leq (2x)^{1/k}} \quad \mathrm{or} \quad 1_{n=1},\quad k+1\leq i\leq 2k,
\end{cases}
\end{align}
and with the difference that $P=\exp((\log \log x)^2)$ and $Q=x^{\varepsilon}$ in \eqref{eq10}.

We apply the Perron formula lemma (Lemma \ref{le_perron}) with the slight modification that $T_0=x^{0.01}$ and $y_1=x^{0.99}$ (the proof works verbatim with these choices). We are then left with showing firstly that
\begin{align}\label{eq13}
\sum_{\substack{x<n\leq x+y_1\\n\in E_2}}1= y_1 \frac{\log \log x}{\log x} + O\left(y_1 \frac{\log \log \log x}{\log x}\right),\quad y_1=x^{0.99},   
\end{align}
and secondly that
\begin{align}\label{eq14}
\int_{T_0}^{x^{1+\varepsilon/10}/H}\left|P\left(\frac{1}{2}+it\right)N_1\left(\frac{1}{2}+it\right)\cdots N_{2k}\left(\frac{1}{2}+it\right)\right|\, dt\ll_{A} x^{1/2}(\log x)^{-A},
\end{align}
where $P(s)=\sum_{P_1<p\leq \min\{2P_1, Q\}}p^{-s}$, $P_1\in [\exp((\log \log x)^2), x^{\varepsilon}]$ and $N_i(s)=\sum_{N_i<n\leq 2N_i}a_i(n)n^{-s}$ with $a_i(n)$ as in \eqref{eq16}. Furthermore, we have the constraint $P_1N_1\cdots N_{2k}\asymp x$. 

To prove \eqref{eq13}, we can for example apply Huxley's prime number theorem, summing first over the $p_2$ variable in the representation $n = p_1 p_2$ with $p_2 \geq p_1$.

For \eqref{eq14}, we split the integration range $[T_0,x^{1+\varepsilon/10}/H]$ into two sets: the set
\begin{align*}
\mathcal{T}_1:=\{t\in [T_0,x^{1+\varepsilon/10}/H]:\,\, |P(1/2+it)|>P_1^{1/2}(\log x)^{-10A}\} \end{align*}
and its complement, which we call $\mathcal{T}_2$. The integral over $\mathcal{T}_2$ can be bounded precisely as in Section~\ref{ssec:arithmetic}, since then we obtain a sufficient pointwise saving in $|P(1/2+it)|$. 

For the integral over $\mathcal{T}_1$, we must proceed differently. To bound
\begin{align}\label{eq15}
\int_{\mathcal{T}_1}\left|P\left(\frac{1}{2}+it\right)N_1\left(\frac{1}{2}+it\right)\cdots N_{2k}\left(\frac{1}{2}+it\right)\right|\, dt,
\end{align}
we first note that if all the $N_i$ satisfy $N_i\leq (2x)^{1/k}$, then we can apply the same argument as in Lemma \ref{le_bilinear} to obtain the desired bound for this (Let $j$ be such that $N_j=\max_{1\leq i\leq 2k}N_i$. Then we group $\{N_1,\ldots, N_{2k}\}\setminus \{N_j\}$ into two almost equal products of size $\in [x^{0.45-\varepsilon/3}, x^{1/2}]$ and apply Cauchy--Schwarz to the Dirichlet polynomials corresponding to these two products and a pointwise bound to $N_j(s)$). Assume then that some $N_{j_0}$ satisfies $N_{j_0}> (2x)^{1/k},$ so that $N_{j_0}(s)$ is a partial sum of $\zeta(s)$ or $\zeta'(s)$. In that case, we bound \eqref{eq15} by
\begin{align*}
\ll (\log x)^{2k}|\mathcal{T}_1|P_1^{1/2}\prod_{i\in [2k]\setminus \{j_0\}}N_i^{1/2}\cdot \sup_{t\in \mathcal{T}_1}\left|N_{j_0}\left(\frac{1}{2}+it\right)\right|.
\end{align*}
By Weyl's method for bounding exponential sums (see e.g. \cite[Corollary 8.6]{iw-kow}) and the fact that $N_{j_0}\gg x^{1/k}$, we have for $t\in \mathcal{T}_1$ the bound $|N_{j_0}(1/2+it)|\ll N_{j_0}^{1/2-\gamma_{0}}$ for some constant $\gamma_{0}>0$. Thus, it suffices to show that
\begin{align*}
|\mathcal{T}_1|=x^{o(1)}    
\end{align*}
to obtain \eqref{eq14} and hence to finish the proof. From a moment estimate given by \cite[Lemma 8]{mr-annals}, we indeed obtain such a bound for $|\mathcal{T}_1|$ (and in fact the stronger bound $|\mathcal{T}_1|\ll \exp(10A \log x / \log \log x)$). This concludes the proof. \qed

\begin{remark}\label{rem_improve}
A similar manoeuvre as in the proof of Theorem \ref{theo_E2} to handle Dirichlet polynomials of length $\exp((\log \log  x)^2)$ would enable us to take the smaller value $P=\exp((\log \log x)^2)$ in the proof of Theorem \ref{theo_mobius}. This then produces the better error term $O(H(\log \log x)^4/\log x)$ in \eqref{eq0}. Similar improvements could be made to our other results. We leave the details to the interested reader.
\end{remark}

\section*{Acknowledgements}
The authors would like to thank Maksym Radziwi{\l\l} for useful suggestions and in particular for pointing out the application to $E_2$-numbers and the reference~\cite{DrappeauTopa}. The authors are also grateful to the referee for useful comments.

The first author was supported by Academy of Finland grant no. 285894. The second author was supported by a Titchmarsh Fellowship of the University of Oxford.
\bibliographystyle{plain}
\bibliography{mobiuspaper}

\end{document}